\documentclass{birkjour}
\usepackage{amsmath,amsfonts,amssymb}
\usepackage[all]{xy}
\usepackage{graphicx}
\usepackage{tikz}
\usepackage{cite}
 \newtheorem{thm}{Theorem} 
 \newtheorem{cor}[thm]{Corollary}
 
 \newtheorem{prop}[thm]{Proposition}
 \theoremstyle{definition}
 \newtheorem{defn}{Definition}
 \theoremstyle{remark}
 \newtheorem{rem}[thm]{Remark}
 \newtheorem{ex}{Example}

\newcommand{\Nn}{\mathbb N}

\renewcommand{\O}{\ensuremath{\mathcal{O}}}

\newcommand{\action}{\curvearrowright}
\newcommand{\raction}{\curvearrowleft}

\renewcommand{\d}{\mathrm d}
\newcommand{\lin}{\mathrm{lin}}

\newcommand{\eps}{\varepsilon}

\newcommand{\rmap}{\longrightarrow}

\newcommand{\X}{\ensuremath{\mathfrak{X}}}
\newcommand{\F}{\ensuremath{\mathcal{F}}}

\newcommand{\Lie}{\mathcal{L}}          
\renewcommand{\L}{\mathbb L}               
\DeclareMathOperator{\Hol}{Hol}         
\DeclareMathOperator{\hol}{hol}         

\newcommand{\G}[1]{G^{(#1)}}
\newcommand{\e}[1]{\eta^{(#1)}}

\newcommand{\al}{\alpha}                
\newcommand{\be}{\beta}

\newcommand{\s}{\mathbf{s}}             
\renewcommand{\t}{\mathbf{t}}           
\renewcommand{\gg}{\mathfrak{g}}        
\renewcommand{\tt}{\mathfrak{t}}        

\newcommand{\tto}{\rightrightarrows}    

\begin{document}

%
%
%
%
%
%
%
%
%

\title{Normal Forms and Lie Groupoid Theory}

\author{Rui Loja Fernandes}

\address{%
Department of Mathematics\\
University of Illinois at Urbana-Champaign\\
1409 W.~Green Street\\
Urbana, IL 61801, USA}

\email{ruiloja@illinois.edu}

\thanks{Supported in part by NSF grant DMS 1308472.}



\date{\today}

\begin{abstract}
In these lectures I discuss the Linearization Theorem for Lie groupoids, and its relation to the various classical linearization theorems for submersions, foliations and group actions. In particular, I explain in some detail the recent metric approach to this problem proposed in \cite{fdh0}.
\end{abstract}

\maketitle
\addtocounter{section}{1}
\section*{Lecture 1: Linearization and Normal Forms}
In Differential Geometry one finds many different normal forms results which share the same flavor. In the last few years we have come to realize that there is more than a shared flavor to many of these results: they are actually instances of the same general result. The result in question is a linearization result for Lie groupoids, first conjectured by Alan Weinstein in \cite{wein1,wein2}. The first complete proof of the linearization theorem was obtained by Nguyen Tien Zung in \cite{zung}. Since then several clarifications and simplifications of the proof, as well as more general versions of this result, were obtained (see \cite{cs,fdh0}). In these lectures notes we give an overview of the current status of the theory. 

The point of view followed here, which was greatly influenced by an ongoing collaboration with Matias del Hoyo \cite{fdh0,fdh1,fdh2}, is that the linearization theorem can be thought of as an Ehresmann's Theorem for a submersion onto a stack. Hence, its proof should follow more or less the same steps as the proof of  the classical Ehresmann's Theorem, which can be reduced to a simple argument using the exponential map of a metric that makes the submersion Riemannian. Although I will not go at all into geometric stacks (see the upcoming paper \cite{fdh2}), I will adhere to the metric approach introduced in~\cite{fdh0}.

Let us recall the kind of linearization theorems that we have in mind. The most basic is precisely the following version of Ehresmann's Theorem:

\begin{thm}[Ehresmann]
Let $\pi:M\to N$ be a \textbf{proper} surjective submersion. Then $\pi$ is locally trivial: for every $y\in N$ there is a neighborhood $y\in U\subset N$, a neighborhood $0\in V\subset T_y N$, and diffeomorphism:
\[
\xymatrix{
 V\times \pi^{-1}(y)\ar[d]_{\text{pr}} \ar[rr]^{\cong}& & \pi^{-1}(U)\subset M\ar[d]^{\pi}\\
V\ar[rr]^{\cong} & & U
}
\]
\end{thm}

One can also assume that there is some extra geometric structure behaving well with respect to the submersion, and then ask if one can achieve ``linearization'' of both the submersion and the extra geometric structure. For example, if one assumes that $\omega\in\Omega^2(M)$ is a closed 2-form such that the pullback of $\omega$ to each fiber is non-degenerate, then one can show that $\pi$ is a locally trivial symplectic fibration (see, e.g., \cite{ms}). We will come back to this later, for now we recall another basic linearization theorem:

\begin{thm}[Reeb]
Let $\F$ be a foliation of $M$ and let $L_0$ be a compact leaf of $\F$ with \textbf{finite holonomy}. Then there exists a saturated neighborhood $L_0\subset U\subset M$, a $\hol(L_0)$-invariant neighborhood $0\in V\subset\nu_{x_0}(L_0)$, and a diffeomorphism:
\[ \xymatrix{
\widetilde{L_0}^h\underset{\hol(L_0)}{\times} V \ar[rr]^{\cong}& &U\subset M}
\]
sending the linear foliation to $\F|_U$.
\end{thm}

Here, $\widetilde{L_0}^h\to L$ denotes the holonomy cover, a $\hol(L_0)$-principal bundle, and the holonomy group $\hol(L_0)$ acts on the normal space $\nu_{x_0}(L_0)$ via the linear holonomy representation. By ``linear foliation" we mean the quotient of the horizontal foliation $\{\widetilde{L_0}^h\times \{t\}$, $t\in\nu_{x_0}(L)\}$.

Notice that this result generalizes Ehresmann's Theorem, at least when the fibers of the submersion are connected: any leaf of the foliation by the fibers of $\pi$ has trivial holonomy so $\hol(L_0)$ acts trivially on the transversal, and then Reeb's theorem immediately yields Ehresmann's Theorem. For this reason, maybe it is not so surprising that the two results are related.

Let us turn to a third linearization result which, in general, looks to be of a different nature from the previous results. It is a classical result from Equivariant Geometry often referred to as the Slice Theorem (or Tube Theorem):

\begin{thm}[Slice Theorem]
Let $K$ be a Lie group acting in a \textbf{proper} fashion on $M$. Around any orbit $\O_{x_0}\subset M$ the action can be linearized: there exist $K$-invariant neighborhoods $\O_{x_0}\subset U\subset M$ and $0_{x_0}\in V\subset \nu_{x_0}(\O_{x_0})$ and a $K$-equivariant diffeomorphism:
\[ \xymatrix{
K\times_{K_{x_0}}V\ar[rr]^{\cong}& & U\subset M}
\]
\end{thm}
Here $K_{x_0}$ acts on the normal space $\nu_{x_0}(\O_{x_0})$ via the normal isotropy representation. 
If the action is locally free then the orbits form a foliation, the isotropy groups $K_x$ are finite and $\hol(\O_x)$ is a quotient of $K_x$. Moreover, the action of $K_{x}$ on a slice descends to the linear holonomy action of $\hol(\O_x)$. The slice theorem is then a special case of the Reeb stability theorem. However, in general, the isotropy groups can have positive dimension and the two results look apparently quite different.

Again, both in the case of foliations and in the case of group actions, we could consider extra geometric structures (e.g., a metric or a symplectic form) and ask for linearization taking into account this extra geometric structure. One can find such linearization theorems in the literature (e.g., the local normal form theorem for Hamiltonian actions \cite{gs}). Let us mention one such recent result from Poisson geometry, due to Crainic and Marcut \cite{cm}:

\begin{thm}[Local normal form around symplectic leaves]
Let $(M,\pi)$ be a Poisson manifold and let $S\subset M$ be a compact symplectic leaf. If the Poisson homotopy bundle $G\action P\to S$ is a smooth \textbf{compact} manifold with vanishing second de Rham cohomology group, then there is a neighborhood $S\subset U\subset M$, and a Poisson diffeomorphism:
\[ \phi:(U,{\pi|}_U)\to (P\times_G \gg,\pi^\lin).\] 
\end{thm}

We will not discuss here the various terms appearing in the statement of this theorem, referring the reader to the original work \cite{cm}. However, it should be clear that this result has the same flavor as the previous ones: some compactness type assumption around a leaf/orbit leads to linearization or a normal form of the geometric structure in a neighborhood of the leaf/orbit.

Although all these results have the same flavor, they do look quite different. Moreover, the proofs that one can find in the literature of these linearization results are also very distinct. So it may come as a surprise that they are actually just special cases of a very general linearization theorem.

In order to relate all these linearization theorems, and to understand the significance of the assumptions one can find in their statements, one needs a language where all these results fit into the same geometric setup. This language exists and it is a generalization of the usual Lie theory from groups to groupoids. We will recall it in the next lecture. After that, we will be in shape to state the general linearization theorem and explain how the results stated before are special instances of it.

\addtocounter{section}{1}
\section*{Lecture 2: Lie Groupoids}

In this Lecture we provide a quick introduction to Lie groupoids and Lie algebroids. We will focus mostly on some examples which have special relevance to us. A more detailed discussion, along with proofs, can be found in \cite{cf}. Let us start by recalling: 

\begin{defn}
A \textbf{groupoid} is a small category where all arrows are invertible.
\end{defn}

Let us spell out this definition. We have a \textbf{set of objects} $M$, and a \textbf{set of arrows} $G$. For each arrow $g\in G$ we can associate its source $\s(g)$ and its target $\t(g)$, resulting in two maps $\s,\t:G\to M$. We also write $g: x\rmap y$ for an arrow with source $x$ and target $y$. 

For any pair of composable arrows we have a \textbf{product} or composition map:
\[ m: \G2\to G, (g,h)\mapsto gh. \]
In general, we will denote by $\G{n}$ the set of $n$ strings of composable arrows:
\[ \G{n}:=\{(g_1,\dots,g_n): \s(g_i)=\t(g_{i+1})\}.\]
The multiplication satisfies the \emph{associativity property}:
\[ (gh)k=g(hk), \quad \forall (g,h,k)\in\G3. \]

For each object $x\in M$ there is an identity arrow $1_x$ and the \emph{identity property} holds:
\[ 1_{\t(g)}g=g=g 1_{\s(g)},\quad \forall g\in G. \]
It gives rise to an \textbf{identity section} $u:M\to G$, $x\mapsto 1_x$. 

For each arrow $g\in G$ there is an inverse arrow $g^{-1}\in G$, for which the \emph{inverse property} holds:
\[ g g^{-1}=1_{\t(g)},\quad g^{-1}g=1_{\s(g)},\quad\forall g\in G.\]
This gives rise to the \textbf{inverse} map $\iota:G\to G$, $g\mapsto g^{-1}$.

\begin{defn}
A \textbf{morphism of groupoids} is a functor $\F:G\to H$.
\end{defn}

This means that we have a map $\F: G\to H$ between the sets of arrows and a map $f: M\to N$ between the sets of objects, making the following diagram commute:
\[
\xymatrix{
 G \ar@<0.25pc>[d]^{\t} \ar@<-0.25pc>[d]_{\s} \ar[r]^{\F} & H \ar@<0.25pc>[d]^{\t}\ar@<-0.25pc>[d]_{\s} \\
M \ar[r]_{f} & N}
\]
such that $\F(gh)=\F(g)\F(h)$ if $g, h\in G$ are composable, and $\F(1_x)= 1_{f(x)}$ for all $x\in M$.

We are interested in groupoids and morphisms of groupoids in the smooth category:

\begin{defn}
A \textbf{Lie groupoid} is a groupoid $G\tto M$ whose spaces of arrows and objects are both manifolds, the structure maps $\s,\t,u,m,i$ are all smooth maps and such that $\s$ and $\t$ are submersions. A \textbf{morphisms of Lie groupoids} is a morphism of groupoids for which the underlying map $\F: G\to H$ is smooth.
\end{defn}

Before we give some examples of Lie groupoids, let us list a few basic properties. 
\begin{itemize}
\item the unit map $u:M\to G$ is an embedding and the inverse $\iota:G\to G$ is a diffeomorphism.

\item The source fibers are embedded submanifolds of $G$ and  \textbf{right multiplication by $g: x\rmap y$} is a diffeomorphism between $\s$-fibers
\[ R_{g}: \s^{-1}(y)\rmap \s^{-1}(x),\quad h\mapsto hg .\]

\item The target fibers are embedded submanifolds of $G$ and \textbf{left multiplication by  $g: x\rmap y$} is a diffeomorphism between $\t$-fibers:
\[ L_{g}: \t^{-1}(x)\rmap \t^{-1}(y),\quad h\mapsto gh .\]

\item The \textbf{isotropy group at $x$}:
\[ G_x=\s^{-1}(x)\cap\t^{-1}(x).\]
is a Lie group.

\item The \textbf{orbit through $x$}: 
\[ \O_x:=\t(\s^{-1}(x))=\{y\in M: \exists\ g: x\rmap y\} \]
is a regular immersed (possibly disconnected) submanifold.

\item The map $\t:\s^{-1}(x)\to \O_x$ is a principal $G_x$-bundle.
\end{itemize}

The connected components of the orbits of a groupoid $G\tto M$ form a (possibly singular) foliation of $M$. The set of orbits is called the \textbf{orbit space} and denoted by $M/G$. The quotient topology makes the natural map $\pi:M\to M/G$ into an open, continuous map. In general, there is no smooth structure on $M/G$ compatible with the quotient topology, so $M/G$ is a singular space. One may think of a groupoid as a kind of atlas for this orbit space.

There are several classes of Lie groupoids which will be important for our purposes. A Lie groupoid $G\tto M$ is called \textbf{source $k$-connected} if the $s$-fibers $\s^{-1}(x)$ are $k$-connected for every $x\in M$. When $k=0$ we say that $G$ is a \textbf{$s$-connected groupoid}, and when $k=1$ we say that $G$ is a \textbf{$s$-simply connected groupoid}. We call the groupoid \textbf{\'etale} if $\dim G=\dim M$, which is equivalent to requiring that the source or target map be a local diffeomorphism. The map $(\s, \t): G\to M\times M$  is sometimes called the \textbf{anchor of the groupoid} and a groupoid is called \textbf{proper} if the anchor $(\s, \t): G\to M\times M$ is a proper map. In particular, when the source map is proper, we call the groupoid \textbf{s-proper}. We will see later that proper and s-proper groupoids are, in some sense, analogues of compact Lie groups.

\begin{ex}
Perhaps the most elementary example of a Lie groupoid is the \textbf{unit groupoid} $M\tto M$ of any manifold $M$: for each object $x\in M$, there is exactly one arrow, namely the identity arrow. More generally, given an open cover $\{U_i\}$ of $M$ one constructs a \textbf{cover groupoid} $\bigsqcup_{i,j\in I} U_i\cap U_j\tto \bigsqcup_{i\in I} U_i$: we picture an arrow as $(x,i)\rmap (x,j)$ if $x\in U_i\cap U_j$. The structure maps should be obvious. In both these examples, the orbit space coincides with the original manifold $M$ and the isotropy groups are all trivial. These are all examples of proper, \'etale, Lie groupoids. If the open cover is not second countable, then the spaces of arrows and objects are not second countable manifolds. In Lie groupoid theory one sometimes allows manifolds which are not second countable. However, we will always assume manifolds to be second countable.
\end{ex}

\begin{ex}
Another groupoid that one can associate to a manifold is the \textbf{pair groupoid} $M\times M\tto M$: an ordered pair $(x,y)$ determines an arrow $x\rmap y$. This is an example of \textbf{transitive groupoid}, i.e., a groupoid with only one orbit. More generally, if $K\action P\to M$ is a principal $K$-bundle, then $K$ acts on the pair groupoid $P\times P\tto P$ by groupoid automorphisms and the quotient $P\times_K P\tto M$ is a transitive groupoid called the \textbf{gauge groupoid}. Note that for any $x\in M$ the isotropy group $G_x$ is isomorphic to $K$ and the principal $G_x$-bundle $\t:\s^{-1}(x)\to M$ is isomorphic to the original principal bundle. Conversely, it is easy to see that any transitive groupoid $G\tto M$ is isomorphic to the gauge groupoid of any of the principal $G_x$-bundles $\t:\s^{-1}(x)\to M$. It is easy to check that that the gauge groupoid associated with a principal $K$-bundle $P\to M$ is proper if and only if $K$ is a compact Lie group. Moreover, it is s-proper (respectively, source $k$-connected) if and only if $P$ is compact (respectively, $k$-connected).
\end{ex}

\begin{ex}
To any surjective submersion $\pi:M\to N$ one can associate the \textbf{submersion groupoid} $M\times_N M\tto M$. This is a subgroupoid of the pair groupoid $M\times M\tto M$, but it fails to be transitive if $N$ has more than one point: the orbits are the fibers of $\pi:M\to N$, so the orbit space is precisely $N$. The isotropy groups are all trivial. The submersion groupoid is always proper and it is s-proper if and only if $\pi$ is a proper map. 
\end{ex}

\begin{ex}
Let $\F$ be a foliation of a manifold $M$. We can associate to it the \textbf{fundamental groupoid} $\Pi_1(\F)\tto M$, whose arrows correspond to foliated homotopy classes of paths (relative to the end-points). Given such an arrow $[\gamma]$, the source and target maps are $\s([\gamma])=\gamma(0)$ and $\t([\gamma])=\gamma(1)$, while multiplication corresponds to concatenation of paths. One can show that $\Pi_1(\F)$ is indeed a manifold, but it may fail to be Hausdorff. In fact, this groupoid is Hausdorff precisely when $\F$ has no vanishing cycles. In Lie groupoid theory, one often allows the total space of a groupoid to be non-Hausdorff, while $M$ and the source/target fibers are always assumed to be Hausdorff. On the other hand, we will always assume manifolds to be second countable. 

Another groupoid one can associate to a manifold is the \textbf{holonomy groupoid} $\Hol(\F)\tto M$, whose arrows correspond to holonomy classes of paths. Again, one can show that $\Hol(\F)$ is a manifold, but it may fail to be Hausdorff. In general, the fundamental groupoid and the holonomy groupoid are distinct, but there is an obvious groupoid morphism $\F:\Pi_1(\F)\to \Hol(\F)$, which to a homotopy class of a path associates the holonomy class of the path (recall that the holonomy only depends on the homotopy class of the path). This map is a local diffeomorphism. 

It should be clear that the leaves of these groupoids coincide with the leaves of $\F$ and that the isotropy groups of $\Pi_1(\F)$ (respectively, $\Hol(\F)$) coincide with the fundamental groups (respectively, holonomy groups) of the leaves. It is easy to check also that $\Pi_1(\F)$ is always $s$-connected. One can show that $\Pi_1(\F)$ (respectively, $\Hol(\F)$) is s-proper if and only if the leaves of $\F$ are compact and have finite fundamental group (respectively, finite holonomy group). In general, it is not so easy to give a characterization in terms of $\F$ of when these groupoids are proper. 
\end{ex}

\begin{ex}
Let $K\action M$ be a smooth action of a Lie group $K$ on a manifold $M$. The associated \textbf{action groupoid} $K\ltimes M\tto M$ has arrows the pairs $(k,x)\in K\times M$, source/target maps given by $\s(k,x)=x$ and $\t(k,x)=kx$, and composition:
\[ (k_1,y)(k_2,x)=(k_1k_2,x),\quad\text{if }y=k_1x. \]
The isotropy groups of this action are the stabilizers $K_x$ and the orbits coincide with the orbits of the action. The action groupoid is proper (respectively, s-proper) precisely when the action is proper (respectively, $K$ is compact). Moreover, this groupoid is source $k$-connected if and only if $K$ is $k$-connected.
\end{ex}

Lie groupoids, just like Lie groups, have associated infinitesimal objects known as \emph{Lie algebroids}:

\begin{defn}
A \textbf{Lie algebroid} is a vector bundle $A\to M$ together with a Lie bracket $[~,~]_A:\Gamma(A)\times\Gamma(A)\to \Gamma(A)$ on the space of sections and a bundle map $\rho_A:A\to TM$, such that the following Leibniz identity holds:
\begin{equation}\label{eq:Leibniz}
[\al,f\be]_A=f[\al,\be]_A+\rho_A(X)(f) Y,
\end{equation}
for all $f\in C^\infty(M)$ and $\al,\be\in\Gamma(A)$.
\end{defn}

Given a Lie groupd $G\tto M$ the associated Lie algebroid is obtained as follows. One lets $A:=\ker\d_M\s$ be the vector bundle whose fibers consists of the tangent spaces to the s-fibers along the identity section. Then sections of $A$ can be identified with right-invariant vector fields on the Lie groupoid, so the usual Lie bracket of vector fields induces a Lie bracket on the sections. The anchor is obtained by restricting the differential of the target, i.e., $\rho_A:=\d\t|_A$.

There is a Lie theory for Lie groupoids/algebroids analogous to the usual Lie theory for Lie groups/algebras. There is however one big difference: Lie's Third Theorem fails and there are examples of Lie algebroids which are not associated with a Lie groupoid (\cite{cf2}). We shall not give any more details about this correspondence since we will be working almost exclusively at the level of Lie groupoids. We refer the reader to \cite{cf} for a detailed discussion of Lie theory in the context of Lie groupoids and algebroids.

\addtocounter{section}{1}
\section*{Lecture 3: The Linearization Theorem}

For  a Lie groupoid $G\tto M$ a submanifold $N\subset M$ is called \textbf{saturated} if it is a union of orbits. Our aim now is to state the linearization theorem which, under appropriate assumptions, gives a normal form for the Lie groupoid in a neighborhood of a saturated submanifold. First we will describe this local normal form, which depends on some standard constructions in Lie groupoid theory.

Let $G\tto M$ be a Lie groupoid. The \textbf{tangent Lie groupoid} $TG\tto TM$ is obtained by applying the tangent functor: hence, the spaces of arrows and objects are the tangent bundles to $G$ and to $M$, the source and target maps are the differentials $\d\s,\d\t:TG\to TM$, the multiplication is the differential $\d m:(TG)^2\equiv T\G2\to TG$, etc.

Assume now that $S\subset M$ is a saturated submanifold. An important special case to keep in mind is when $S$ consists of a single orbit of $G$. Then we can restrict the groupoid $G$ to $S$:
\[ G_S:=\t^{-1}(S)=\s^{-1}(S), \]
obtaining a Lie subgroupoid $G_S\tto S$ of $G\tto M$. If we apply the tangent functor, we obtain a Lie subgroupoid $TG_S\tto TS$ of $TG\tto TM$. 

\begin{prop}
There is a short exact sequence of Lie groupoids:
\[ 
\xymatrix{
1\ar[r] & TG_S \ar@<0.25pc>[d] \ar@<-0.25pc>[d] \ar[r] & TG \ar@<0.25pc>[d] \ar@<-0.25pc>[d] \ar[r]  & \nu(G_S)\ar@<0.25pc>[d] \ar@<-0.25pc>[d] \ar[r] & 1\\
0\ar[r]  & TS \ar[r] & TM \ar[r] & \nu(S)\ar[r]& 0}
 \]
\end{prop}

There is an alternative description of the groupoid $\nu(G_S)\tto\nu(S)$ which sheds some light on its nature. For each arrow $g:x\rmap y$ in $G_S$ one can define the 
linear transformation: 
\[ T_g:\nu_x(S)\to \nu_y(S), [v] \mapsto [\d_g\t(\tilde v)], \]
where $\tilde v\in T_gG$ is such that $\d_gs(\tilde v)=v$. One checks that this map is independent of the choice of lifting $\widetilde{v}$. Moreover, for any identity arrow one has $T_{1_x}=$id$_{\nu_x(S)}$ and for any pair of composable arrows $(g,h)\in\G2$ one finds:
\[ T_{gh}= T_g\circ T_h. \]
This means that $G_S\tto S$ acts linearly on the normal bundle $\nu(S)\to S$, and one can form the \textbf{action groupoid}:
\[ G_S\ltimes\nu(S)\tto \nu(S). \]
The space of arrows of this groupoid is the fiber product $G_S \times_S \nu(S)$, with source and target maps: $\s(g,v)=v$ and $\t(g,v)=T_g v$.
The product of two composable arrows $(g,v)$ and $(h,w)$ is then given by:
\[ (g,v)(h,w)=(gh,w). \]
One then checks that:

\begin{prop}
The map $\F:\nu(G_S)\to G_S\ltimes\nu(S)$, $v_g\mapsto (g,[\d_g\s(v_g)])$, is an isomorphism of Lie groupoids.
\end{prop}

Finally, let us observe that the restricted groupoid $G_S\tto S$ sits (as the zero section) inside the Lie groupoid $\nu(G_S)\tto\nu(S)$ as a Lie subgroupoid. This justifies introducing the following definition:

\begin{defn}
For a saturated submanifold $S$ of a Lie groupoid $G\tto M$ the local \textbf{linear model} around $S$ is the groupoid $\nu(G_S)\tto\nu(S)$.
\end{defn}

\begin{ex}[Submersions] For the submersion groupoid $G=M\times_N M\tto M$ associated with a submersion $\pi:M\to N$, a fiber $S=\pi^{-1}(z)$ is a saturated submanifold (actually, an orbit) and we find that 
\[ G_S:=\{(x,y)\in M\times M: \pi(x)=\pi(y)=z\}=S\times S\tto S. \]
is the pair groupoid. Notice that the normal bundle $\nu_M(S)$ is naturally isomorphic to the trivial bundle $S\times T_z N$. It follows that the local linear model is the groupoid:
\[ S\times S\times T_z N\tto S\times T_z N, \]
which is the direct product of the pair groupoid $S\times S\to S$ and the identity groupoid $T_zN\tto T_zN$. More importantly, we can view this groupoid as the submersion groupoid of the projection $:S\times T_zN\to T_zN$.
\end{ex}

\begin{ex}[Foliations] Let $\F$ be a foliation of $M$ and let $L\subset M$ be a leaf. The normal bundle $\nu(\F)$ has a natural flat $\F$-connection, which can be described as follows. Given a vector field $Y\in\X(M)$ let us write $\overline{Y}\in\Gamma(\nu(\F))$ for the corresponding section of the normal bundle. Then, if $X\in\X(\F)$ is a foliated vector field, one sets:
\[ \nabla_X\overline{Y}:=\overline{[X,Y]}. \]
One checks that this definition is independent of the choice of representative, and defines a connection $\nabla:\X(\F)\times\Gamma(\nu(\F))\to\Gamma(\nu(\F))$, called the \emph{Bott connection}. The Jacobi identity shows that this connection is flat. 

Given a path $\gamma:I\to L$ in some leaf $L$ of $\F$, parallel transport along $\nabla$ defines a linear map 
\[ \tau_\gamma:\nu_{\gamma(0)}(L)\to \nu_{\gamma(1)}(L). \]
Since the connection is flat, it is clear that this linear map only depends on the homotopy class $[\gamma]$. One obtains a linear action of the fundamental groupoid $\Pi_1(\F)$ on $\nu(\F)$. The restriction of $\Pi_1(\F)$ to the leaf $L$ is the fundamental groupoid of the leaf $L$ and we obtain the linear model for $\Pi_1(\F)$ along the leaf $L$ as the action groupoid:
\[ \Pi_1(L)\ltimes\nu(L)\tto \nu(L). \]
The fundamental groupoid $\Pi_1(L)$ is isomorphic to the gauge goupoid of the universal covering space $\widetilde{L}\to L$, viewed as a principal $\pi_1(L)$-bundle. Moreover, $\nu(L)$ is isomorphic to the associated bundle $\widetilde{L}\times_{\pi_1(L,x)}\nu_x(L)$. It follows that the linear model coincides with the fundamental groupoid of the linear foliation of $\nu(L)=\widetilde{L}\times_{\pi_1(L,x)}\nu_x(L)$.

The linear map $\tau_\gamma$ only depends on the holonomy class of $\gamma$, since this maps coincides with the linearization of the holonomy action along $\gamma$. For this reason, there is a similar description of the linear model of the holonomy groupoid  $\Hol(\F)$ along the leaf $L$ as an action groupoid:
\[ \Hol(\F)_L\ltimes\nu(L)\tto \nu(L). \]
Also, the holonomy groupoid is isomorphic to the gauge goupoid of the holonomy cover $\widetilde{L}^h\to L$, viewed as a principal $\hol(L)$-bundle and we can also describe the normal bundle $\nu(L)$ as the associated bundle $\widetilde{L}^h\times_{\hol(L,x)}\nu_x(L)$. The underlying foliation of this linear model is still the linear foliation of $\nu(L)$. However, the linear model now depends on the germ of the foliation around $L$, i.e., on the non-linear holonomy. Unlike the case of the fundamental groupoid, the knowledge of the Bott connection is not enough to build this linear model.
\end{ex}

\begin{ex}[Group actions] Let $K$ be a Lie group that acts on a manifold $M$ and let $K_x$ be the isotropy group of some $x\in M$. For each $k\in K_x$, the map $\Phi_k: M\to M$, $y\mapsto ky$, fixes $x$ and maps the orbit $\O_x$ to itself. Hence, $\d_x\Phi_k$ induces a linear action of $K_x$ on the normal space $\nu_x(\O_x)$, called the \emph{normal isotropy representation}. Using this representation, it is not hard to check that we have a vector bundle isomorphism:
\[ \xymatrix@R=10pt{\nu(\O_x)\ar[rr]^{\simeq}\ar[dr] & &K\times_{K_x}\nu_x(\O_x)\ar[dl]\\ &\O_x} \]
where the action $K_x\action K\times \nu_x(O_x)$ is given by $k(g,v):=(g k^{-1},kv)$. Moreover, one has an action of $K$ on $\nu(\O_x)$, which under this isomorphism corresponds to the action:
\[ K\action K\times_{K_x}\nu_x(O_x), \quad k [(k',v)]=[(kk',v)]. \]

Now consider the action Lie groupoid $K\ltimes M\to M$. One checks that the local linear model around the orbit $\O_x$ is just the action Lie groupoid $K\ltimes \nu(\O_x)\tto \nu(\O_x)$, which under the isomorphism above corresponds to the action groupoid:
\[ K\ltimes (K\times_{K_x}\nu_x(O_x))\tto K\times_{K_x}\nu_x(O_x). \]
\end{ex}

As we have already mentioned above, the Linearization Theorem states that, under appropriate conditions, the groupoid is locally isomorphic around a saturated submanifold to its local model. In order to make precise the expression ``locally isomorphic'' we introduce the following definition:

\begin{defn}
Let $G\tto M$ be a Lie groupoid and $S\subset M$ a saturated submanifold. A {\bf groupoid neighborhood} of $G_S\tto S$ is a pair of open sets $U\supset S$ and $\widetilde{U}\supset G_S$ such that $\widetilde{U}\tto U$ is a subgroupoid of $G\tto M$. A groupoid neighborhood $\widetilde{U}\tto U$ is said to be \textbf{full} if $\widetilde{U}=G_U$.
\end{defn}

Our first version of the linearization theorem reads as follows:

\begin{thm}[Weak linearization]
\label{thm:weak:linear}
Let $G\tto M$ be a Lie groupoid with a 2-metric $\e2$. Then $G$ is \textbf{weakly linearizable} around any saturated submanifold $S\subset M$: there are groupoid neighborhoods $\widetilde{U}\tto U$ of $G_S\tto S$ in $G\tto M$ and $\widetilde{V}\tto V$ of $G_S\to S$ in the local model $\nu(G_S)\tto\nu(S)$, and an isomorphism of Lie groupoids:
\[ (\widetilde{U}\tto U)\overset\phi\cong (\widetilde{V}\tto V), \]
which is the identity on $G_S$.
\end{thm}

A 2-metric is a special Riemannian metric in the space of composable arrows $\G2$. We will discuss them in detail in the next lecture. For now we remark that for a \emph{proper} groupoid 2-metrics exist and, moreover, every groupoid neighborhood contains a full groupoid neighborhood. Hence:

\begin{cor}[Linearization of proper groupoids]
Let $G\tto M$ be a proper Lie groupoid. Then $G$ is \textbf{linearizable} around any saturated submanifold $S\subset M$: there exist open neighborhoods $S\subset U\subset M$ and $S\subset V\subset \nu(S)$ and an isomorphism of Lie groupoids
\[ (G_U\tto U)\overset\phi\cong (\nu(G_S)_V\tto V), \]
which is the identity on $G_S$.
\end{cor}

This corollary does not yet yield the various linearization results stated in Lecture 1. The reason is that the assumptions do not guarantee the existence of a saturated neighborhood $S\subset U\subset M$. This can be realized for s-proper groupoids, since for such groupoids every neighborhood $U$ of a saturated embedded submanifold $S$, contains a saturated neighborhood of $S$.

\begin{cor}[Invariant linearization of s-proper groupoids]
Let $G\tto M$ be an s-proper Lie groupoid. Then $G$ is \textbf{invariantly linearizable} around any saturated submanifold $S\subset M$: there exist saturated open neighborhoods $S\subset U\subset M$ and $S\subset V\subset \nu(S)$ and an isomorphism of Lie groupoids
\[ (G_U\tto U)\overset\phi\cong (\nu(G_S)_V\tto V),\]
which is the identity on $G_S$.
\end{cor}

\begin{ex}
For a proper submersion $\pi:M\to N$ the associated submersion groupoid $M\times_N M\tto M$ is s-proper. Using the description of the local model that we gave before, it follows that for any fiber $\pi^{-1}(z)$ there is a saturated open neighborhood $\pi^{-1}(z)\subset U\subset M$ where the submersion if locally isomorphic to the trivial submersion $\pi^{-1}(z)\times V\to V$, for some open neighborhood $0\in V\subset T_z N$. Hence, we recover the classical Ehresmann's Theorem.
\end{ex}

\begin{ex}
Let $\F$ be a foliation of $M$ whose leaves are compact with finite holonomy. Then the holonomy groupoid $\Hol(\F)\tto M$ is s-proper. Using the description of the local model given before, it follows that for any leaf $L$ there is a saturated neighborhood $L\subset U\subset M$ where the canonical foliation is isomorphic to the linear foliation of $\nu(L)=\widetilde{L}^h\times_{\hol(L,x)} V$, for some $\hol(L,x)$-invariant, neighborhood $0\in V\subset \nu_x(L)$. Hence, we recover the Local Reeb Stability Theorem.
\end{ex}

\begin{ex}
Let $K\times M\to M$ be an action of a \emph{compact} Lie group. Then the the action groupoid is $s$-proper and we obtain invariant linearization. From the description of the local model, it follows that for any orbit $\O_x$ there is a saturated open neighborhood $\O_x\subset U\subset M$ and a $K$-equivariant isomorphism $U\simeq K\times_{K_x} V$, where $0\in V\subset \nu_x(\O_x)$ is a $K_x$-invariant neighborhood. Hence, we recover the slice theorem for actions of compact groups. For a general proper action, the results above only give weak linearization, which does not allow to deduce the slice theorem. However, due to the particular structure of the action groupoid, every orbit has a saturated neighborhood and one has a uniform bound for the injectivity radius of the 2-metric. This gives invariant linearization and leads to the Slice Theorem for \emph{any} proper action.
\end{ex}

\begin{ex}
Let $(M,\pi)$ be a Poisson manifold. The cotangent bundle $T^*M$ has a natural Lie algebroid structure with anchor $\rho:T^*M\to TM$ given by contraction by $\pi$ and Lie bracket on sections (i.e., 1-forms):
\[ [\al,\be]:=\Lie_{\rho(\al)}\be-\Lie_{\rho(\be)}\al-\d \pi(\al,\be). \]
In general, this Lie algebroid fails to be integrable. However,  under the assumptions of the local normal form theorem stated in Lecture 1, this groupoid is integrable, and then its source 1-connected integration is an s-proper Lie groupoid whose orbits are the symplectic leaves. This groupoid can then be linearized around a symplectic leaf, but this linearization \emph{does not} yet yield the local canonical form for the Poisson structure. 

It turns out that the source 1-connected integration is a symplectic Lie groupoid, i.e., there is a symplectic structure on its space of arrows which is compatible with multiplication. One can apply a Moser type trick to further bring the symplectic structure on the local normal form to a canonical form, which then yields the canonical form of the Poisson structure. The details of this approach, which differ from the original proof of the canonical form due to Crainic and Marcut, can be found in \cite{cm}.
\end{ex}

\addtocounter{section}{1}
\section*{Lecture 4: Groupoid Metrics and Linearization}

Let us recall that a submersion $\pi:(M,\eta)\to N$ is called a \textbf{Riemannian submersion} if the fibers are equidistant. The base $N$ gets an induced metric $\pi_*\eta$ for which the linear maps 
$\d_x\pi:(\ker\d_x\pi)^\perp \to T_{\pi(x)} N$, $x\in M$, are all isometries. More generally, a (possibly singular) foliation in a Riemannian manifold $M$ is called a \textbf{Riemannian foliation} if the leaves are equidistant. This is equivalent to the following property: any geodesic which is perpendicular to one leaf at some point stays perpendicular to all leaves that it intersects. 

A simple proof of Ereshman's Theorem can be obtained by choosing a metric on the total space of the submersion $\pi:M\to N$, that makes it into a Riemannian submersion. A partition of unit argument shows that this is always possible. Then the linearization map is just the exponential map of the normal bundle of a fiber, which maps onto a saturated neighborhood of the fiber, provided the submersion is proper. We will see that a proof similar in spirit also works for the general linearization theorem (Theorem \ref{thm:weak:linear}). 

Let $G\tto M$ be a Lie groupoid. Like any category, $G$ has a \textbf{simplicial model}:
\[ 
\xymatrix@1{ \dots \ar@<0.60pc>[r]\ar@<0.30pc>[r] \ar[r] \ar@<-0.30pc>[r]\ar@<-0.60pc>[r]& \G{n}  \ar@<0.45pc>[r]\ar@<0.15pc>[r]\ar@<-0.15pc>[r]\ar@<-0.45pc>[r] &\cdots  \ar@<0.45pc>[r]\ar@<0.15pc>[r]\ar@<-0.15pc>[r]\ar@<-0.45pc>[r]&\G2 \ar@<0.30pc>[r] \ar[r] \ar@<-0.30pc>[r]& \G1 \ar@<0.2pc>[r]\ar@<-0.2pc>[r] & \G0\ [\ar[r]& \G0/\G1]}
\] 
where, for each $n\in\Nn$, the \textbf{face maps} $\eps_i:\G{n}\to\G{n-1}$, $i=0,\dots,n$ and the \textbf{degeneracy maps} $\delta_i:\G{n}\to\G{n+1}$, $i=1,\dots,n$, are defined as follows: for a $n$-string of composable arrows the $i$-th face map associates the $(n-1)$-string of composable arrows obtained by omitting the $i$-object:
\[ \eps_i\left( \xymatrix{\cdot & \ar@/_/[l]_{g_1} \cdot \ar@{..}[r]&\cdot & \ar@/_/[l]_{g_n} \cdot }\right)=
\xymatrix{\cdot & \ar@/_/[l]_{g_1} \cdot \ar@{..}[r]& &\ar@/_/[l]_{g_ig_{i+1}} \cdot \ar@{..}[r]&\cdot & \ar@/_/[l]_{g_n} \cdot }\]
while the $i$-th degeneracy map inserts into a $n$-string of composable arrows an identity at the $i$-th entry:
\[ \delta_i\left( \xymatrix{\cdot & \ar@/_/[l]_{g_1} \cdot \ar@{..}[r]&\cdot & \ar@/_/[l]_{g_n} \cdot }\right)=
\xymatrix{\cdot & \ar@/_/[l]_{g_1} \cdot \ar@{..}[r]&\cdot\ar@(ur,ul)[]_{1_{x_i}} \ar@{..}[r]&\cdot & \ar@/_/[l]_{g_n} \cdot }\]
For a Lie groupoid there is, additionally, a natural action of $S_{n+1}$ on $\G{n}$: for a string of n-composable arrows $(g_1,\dots,g_n)$ choose $(n+1)$ arrows $(h_0,\dots,h_n)$, 
all with the same source, so that:
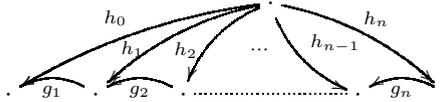
\[ \xymatrix{
  & &  & \cdot \ar@/_/[dlll]_{h_0\ \ }^{h_1\ \ \ }\ar@/_/[dll]^{h_2}  \ar@/_/[dl]^{\ \ \ \cdots} \ar@/_/[dr]^{h_{n-1}}  \ar@/^/[drr]^{h_n}\\
\cdot & \ar@/_/[l]^{g_1} \cdot &\ar@/_/[l]^{g_2}\cdot \ar@{..}[rr]& &\cdot & \ar@/_/[l]^{g_n} \cdot} \]
Then the $S_{n+1}$-action on the arrows $(h_0,\dots,h_n)$ by permutation gives a well defined $S_{n+1}$-action on $\G{n}$. Notice that this action permutes the face maps $\eps_i$, since there are maps $\phi_i:S_{n+1}\to S_n$ such that:
\[ \eps_i\circ \sigma=\phi_i(\sigma)\circ \eps_{\sigma(i)}. \]

\begin{defn}
A \textbf{$n$-metric} ($n\in\Nn$) on a groupoid $G\to M$ is a Riemannian metric $\e{n}$ on $\G{n}$ which is $S_{n+1}$-invariant and for which all the face maps  $\eps_i:\G{n}\to\G{n-1}$ are Riemmanian submersions.
\end{defn}

Actually, it is enough in this definition to ask that one of the face maps is a Riemannian submersion: the assumption that the action of $S_{n+1}$ is by isometries implies that if one face map is a Riemannian submersion then all the face maps are Riemannian submersions. 

For any $n\ge 1$, the metrics induced on $\G{n-1}$ by the different face maps $\eps_i:\G{n}\to\G{n-1}$ coincide, giving a well defined metric $\e{n-1}$, which is a $(n-1)$-metric. Obviously, one can repeat this process, so that a $n$-metric on $\G{n}$ determines a $k$-metric on $\G{k}$ for all $0\le k\le n$.  

\begin{ex}[0-metrics]
When $n=0$ we adopt the convention that $\e0$ is a metric on $M=\G0$ which makes the orbit space a Riemannian foliation and is invariant under the action of $G$ on the normal space to the orbits, i.e., such that each arrow $T_g$ acts by isometries on $\nu(\O)$. Such 0-metrics were studied by Pflaum, Posthuma and Tang in \cite{ppt}, in the case of proper groupoids. One can think of such metrics as determining a metric on the (possibly singular) orbit space $M/G$. Indeed, it is proved in \cite{ppt} that the distance on the orbit space determined by $\e0$ enjoys some nice properties.
\end{ex}

\begin{ex}[1-metrics]
A 1-metric is just a metric $\e1$ on the space of arrows $G=\G1$ for which the source and target maps are Riemannian submersions and inversion is an isometry. These metrics were studied by Gallego \emph{et al.} in \cite{gghr}. A 1-metric induces a 0-metric $\e0$ on $M=\G0$, for which the orbit foliation is Riemannian. 
\end{ex}

\begin{ex}[2-metrics]
A $2$-metric is a metric $\e2$ in the space of composable arrows $\G2$, which is invariant under the $S_3$-action generated by the involution $(g_1,g_2)\mapsto (g_2^{-1},g^{-1}_1)$ and the 3-cycle  $(g_1,g_2)\mapsto ((g_1g_2)^{-1},g_1)$, and for which multiplication is a Riemannian submersion. Hence, a 2-metric on $\G2$ induces a 1-metric on $\G1$. These metrics were introduced in \cite{fdh0}. 
\end{ex}

Notice that the first 3 stages of the nerve 
\[ \xymatrix@1{ \G2 \ar@<0.5pc>[r]^{\pi_1} \ar[r]|m \ar@<-0.5pc>[r]_{\pi_2} & G \ar@<0.25pc>[r]^s \ar@<-0.25pc>[r]_t & M} \]
completely determine the remaining $\G{n}$, for $n\ge 3$. Hence, one should expect that $n$-metrics, for $n\ge 3$, are determined by their $2$-metrics. In fact, one has the following properties, whose proof can be found in \cite{fdh1,fdh2}:
\begin{itemize}
\item there is at most one 3-metric inducing a given 2-metric and every 3-metric has a unique extension to an $n$-metric for every $n\ge 3$.
\item there are examples of groupoids which admit an $n$-metric, but do not admit a $n+1$-metric, for $n=0,1,2$. 
\item uniqueness fails in low degrees: one can have, e.g., two different 2-metrics on $\G2$ inducing the same 1-metric on $\G1$.
\end{itemize}

The geometric realization of the nerve of a groupoid $G\tto M$ is usually denoted by $BG$, can be seen as the classifying space of principal $G$-bundles (see \cite{h}). Two Morita equivalent groupoids $G_1\tto M_1$ and $G_2\tto M_2$ (see \cite{survey} or \cite{mm}), give rise to homotopy equivalent spaces $BG_1$ and $BG_2$. An alternative point of view, is to think of $BG$ as a geometric stack with atlas $G\tto M$, and two atlas represent the same stack if they are Morita equivalent groupoids. The following result shows that one may think of a $n$-metric as a metric in $BG$:

\begin{prop}[\cite{fdh1,fdh2}]
If $G\tto M$ and $H\tto N$ are Morita equivalent groupoids, then $G$ admits a $n$-metric if and only if $H$ admits a $n$-metric.
\end{prop}

In fact, it is shown in \cite{fdh1,fdh2} that it is possible to ``transport'' a $n$-metric via a Morita equivalence. This constructions depends on some choices, but the \emph{transversal component} of the $n$-metric is preserved. We refer to those references for a proof and more details.

Since a $n$-metric determines a $k$-metric, for all $0\le k\le n$, a necessary condition for a groupoid $G\tto M$ to admit a $n$-metric is that the orbit foliation can be made Riemannian. This places already some strong restrictions on the class of groupoids admitting a $n$-metric. So one may wonder when can one find such metrics. One important result in this direction is that any proper groupoid admits such a metric:

\begin{thm}
A proper Lie groupoid $G\tto M$ admits $n$-metrics for all $n\ge 0$.
\end{thm}

\begin{proof}
The proof uses the following trick, called in \cite{fdh0} the \textbf{gauge trick}.  For each $n$, consider the manifold $G^{[n]}\subset G^n$ of  $n$-tuples of arrows with the same source, and the map 
\[ \pi^{(n)}:G^{[n+1]}\to\G{n},\quad (h_0,h_1\dots,h_{n})\longmapsto (h_0h_1^{-1},h_1h_2^{-1},\dots,h_{n-1}h_{n}^{-1}).\] 
The fibers of $\pi^{(n)}$ coincide with the orbits of the right-multiplication action,
$$
\begin{matrix}
G^{[n+1]}\raction G\\
(h_0,\dots,h_{n})\cdot k = (h_0k,\dots,h_{n}k).
\end{matrix}
\qquad \qquad
\begin{matrix}
\xymatrix@R=10pt{\bullet & & \\ \bullet & \bullet \ar[lu]_{h_0} \ar[l]\ar[ld]^{h_{n}} & \bullet \ar[l]^k \\ \bullet }\end{matrix}
$$
and this action is free and proper, hence defining a principal $G$-bundle. The strategy is to define a metric on $G^{[n+1]}$ in a such way that $\pi^{(n)}$ becomes a Riemannian submersion, and that the resulting metric on $\G{n}$ is a $n$-metric.

The group $S_{n+1}$ acts on the manifold $G^{[n+1]}$ by permuting its coordinates, and this action covers the action $S_{n+1}\action\G{n}$, so the map $\pi^{(n)}$ is $S_{n+1}$ equivariant.
On the other hand, there are $(n+1)$ left groupoid actions $G\action G^{[n+1]}$, each consisting in left multiplication on a given coordinate.
$$
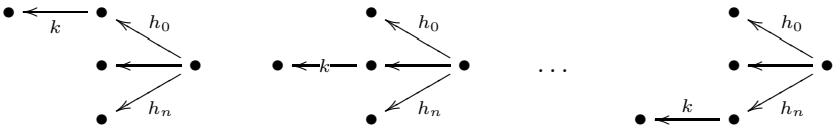
\begin{matrix}
\xymatrix@R=10pt{\bullet & \bullet \ar[l]^k & \\ & \bullet & \bullet \ar[lu]_{h_0} \ar[l] \ar[ld]^{h_{n}}  \\ & \bullet }\end{matrix}
\qquad
\begin{matrix}
\xymatrix@R=10pt{  & \bullet & \\\bullet & \bullet \ar[l]|k & \bullet \ar[lu]_{h_0} \ar[l] \ar[ld]^{h_{n}}  \\ & \bullet }\end{matrix}
\qquad\dots\qquad
\begin{matrix}
\xymatrix@R=10pt{  & \bullet &  \\ & \bullet & \bullet \ar[lu]_{h_0} \ar[l] \ar[ld]^{h_{n}}  \\\bullet & \bullet \ar[l]_k &}\end{matrix}
$$
These left actions commute with the above right action and cover $(n+1)$-principal actions $G\action\G{n}$, with projection the face maps $\eps_i:\G{n}\to\G{n-1}$.

Now for a proper groupoid one can use averaging to construct a metric on $G$ which is invariant under the left action of $G$ on itself by left translations. The product metric on $G^{[n+1]}$ is invariant both under the $(n+1)$ left $G$-actions above and the $S_{n+1}$-action. In general, it will not be invariant under the right $G$-action $G\action G^{[n+1]}\to \G{n}$, but we can average it to obtain a new metric which is invariant under all actions. It follows that the resulting metric on $\G{n}$ is an $n$-metric.
\end{proof}

\begin{rem}
The maps $\pi^{(n)}:G^{[n+1]}\to\G{n}$, $n=0,1,\dots$ that appear in this proof form the simplicial model for the universal principal $G$-bundle $\pi:EG\to BG$. This bundle plays a key role in many different constructions associated with the groupoid (see \cite{fdh1}).
\end{rem}

Still, there are many examples of groupoids, which are not necessarily proper, but admit a $n$-metric. Some are given in the next set of examples. 

\begin{ex}
The unit groupoid $M\tto M$ obviously admits $n$-metrics. The pair groupoid $M\times M\tto M$ and, more generally, the submersion groupoid $M\times_N M\tto M$ associated with a submersion $\pi: M\to N$, are proper so also admit $n$-metrics. For example, if $\eta$ is a metric on $M$ which makes $\pi$ a Riemannian submersion, then one can take $\e0=\eta$, $\e1=p_1^*\eta+p_2^* \eta-p_N^*\eta_N$, 
$\e2=p_1^*\eta+p_2^* \eta+p_3^* \eta-2 p_N^*\eta_N$, etc.
\end{ex}

\begin{ex}
Let $\F$ be a foliation of $M$. Then $\Hol(\F)$ and $\Pi_1(\F)$ admit a $n$-metric if and only if $\F$ can be made into a Riemannian foliation. We already know that if these groupoids admit a $n$-metric, then the underlying foliation, i.e. $\F$, must be Riemannian. Conversely, if $\F$ is Riemannian then $\Hol(\F)$ is a proper groupoid (see, e.g., \cite{m}), so it carries a $n$-metric. Since $\Pi_1(\F)$ is a covering of $\Hol(\F)$, it also admits a $n$-metric. Not that, in general, $\Pi_1(\F)$ does not need to be proper (e.g., if the fundamental group of some leaf is not finite).
\end{ex}

\begin{ex}
Any Lie group admits a $n$-metric. More generally, the action Lie groupoid associated with any isometric action of a Lie group on a Riemannian manifold admits a $n$-metric. This can be shown using a gauge trick, similar to the one used in the proof of existence of a $n$-metric for a proper groupoid sketched above (see \cite{fdh0}). Note that an isometric action does not need not be proper (e.g., if some isotropy group is not compact).
\end{ex}

Finally, we can sketch the proof of the main Linearization Theorem, which we restate now in the following way:

\begin{thm}
\label{thm:main:2}
Let $G\tto M$ be a Lie groupoid endowed with a 2-metric $\eta^{(2)}$, and let $S\subset M$ be a saturated embedded submanifold. Then the exponential map defines a weak linearization of $G$ around $S$.
\end{thm}

\begin{proof}
One choses a neighborhood $S\subset V\subset \nu(S)$ where the exponential map of $\e0$ is a diffeomorphism onto its image, and set
\[ \widetilde V = (\d s)^{-1}(V)\cap (\d t)^{-1}(V) \cap \text{Domain of }\exp_{\e1}\subset \nu(G_S).\]
One shows that $\widetilde{V}\tto V$ is a groupoid neighborhood of $G_S\tto S$ in the linear model $\nu(G_S)\tto\nu(S)$ and that we have a commutative diagram:
\[ 
\xymatrix@1{
\widetilde{V}\times_V \widetilde{V} \ar[rr]^{\exp_{\e2}}\ar@<0.5pc>[d] \ar[d] \ar@<-0.5pc>[d]& & \G2 \ar@<0.5pc>[d] \ar[d] \ar@<-0.5pc>[d] \\
\widetilde{V} \ar[rr]^{\exp_{\e1}} \ar@<0.25pc>[d] \ar@<-0.25pc>[d] & & G\ar@<0.25pc>[d] \ar@<-0.25pc>[d]\\
V \ar[rr]^{\exp_{\e0}}  & &M}
\]
It follows that the exponential maps of $\e1$ and $\e0$ give the desired weak linearization, i.e., a groupoid isomorphism
\[ (\widetilde{V}\tto V)\overset{\exp}\cong (\exp(\widetilde{V})\tto \exp(V)). \]
\end{proof}


\subsection*{Acknowledgments}
I would specially like to thank Matias del Hoyo, for much of the work reported in these notes rests upon our ongoing collaboration \cite{fdh0,fdh1,fdh2}. He and my two PhD students, Daan Michiels and Joel Villatoro, made several very usual comments on a preliminary version of these notes, that helped improving them. I would also like to thank Marius Crainic, Ioan Marcut, David Martinez-Torres and Ivan Struchiner, for many fruitful discussions which helped shape my view on the linearization problem. Finally, thanks go also to the organizers of the XXXIII WGMP, for the invitation to deliver these lectures and for a wonderful stay in Bialowieza.


\begin{thebibliography}{xx}

\bibitem{cf}
M. Crainic and R.L. Fernandes;
Lectures on integrability of Lie brackets;
\emph{Geom. Topol. Monogr.} \textbf{17} (2011), 1-107.

\bibitem{cf2} M.~Crainic and R.~L.~Fernandes;
  Integrability of Lie brackets; 
  \emph{Ann.~of Math.}~(2) \textbf{157} (2003), 575-620.

\bibitem{cs}
M. Crainic and I. Struchiner;
On the Linearization Theorem for proper Lie groupoids;
\emph{Ann. Scient. \'Ec. Norm. Sup.} $4^e$ s\'erie, \textbf{46} (2013), 723-746.

\bibitem{cm}
M. Crainic and I. Marcut;
A Normal Form Theorem around Symplectic Leaves;
\emph{J. Differential Geom.} \textbf{92} (2012), no. 3, 417-461.

\bibitem{survey}
M. del Hoyo;
Lie groupoids and their orbispaces;
\emph{Portugal. Math.} \textbf{70} (2013), 161-209.

\bibitem{fdh0}
M. del Hoyo and R.L. Fernandes;
Riemannian metrics on Lie groupoids;
preprint \emph{arXiv:1404.5989}.

\bibitem{fdh1}
M. del Hoyo and R.L. Fernandes;
Simplicial metrics on Lie groupoids;
\emph{Work in progress}.

\bibitem{fdh2}
M. del Hoyo and R.L. Fernandes;
A stacky version of Ehresmann's Theorem;
\emph{Work in progress}.

\bibitem{gghr}
E. Gallego, L. Gualandri, G. Hector and A. Reventos;
Groupo\"{\i}des Riemanniens;
\emph{Publ. Mat.} \textbf{33} (1989), no. 3, 417--422.

\bibitem{gs}
V. Guillemin and S. Sternberg;
\emph{Symplectic techniques in physics}, Cambridge Univ. Press, Cambridge, 1984.

\bibitem{h}
A.~Haefliger, \emph{Homotopy and integrability}, in Manifolds--Amsterdam 1970 (Proc. Nuffic Summer School), Lecture Notes in Mathematics, Vol. 197 197, Berlin, New York: Springer-Verlag, 133--163.

\bibitem{ms} 
D.~McDuff and D.~Salamon;
  \emph{Introduction to symplectic topology},
  $2^\text{nd}$ edition, Oxford Mathematical Monographs, Oxford University Press, New York, 1998.

\bibitem{mm}
I. Moerdijk and J. Mrcun;
\emph{Introduction to foliations and Lie groupoids};
Cambridge Stud. Adv. Math. 91, Cambridge University Press, Cambridge, 2003.

\bibitem{m}
P. Molino; \emph{Riemannian foliations}, Progress in mathematics vol. 73, Birkh\"auser, Basel, 1988.

\bibitem{ppt}
M. Pflaum, H. Poshtuma and X. Tang;
Geometry of orbit spaces of proper Lie groupoids;
to appear in \emph{J. Reine Angew. Math.}. Preprint \emph{arXiv:1101.0180}.

\bibitem{wein1}
A. Weinstein;
Linearization problems for Lie algebroids and Lie groupoids,
\emph{Lett. Math. Phys.}  \textbf{52} (2000), 93-102.

\bibitem{wein2}
A. Weinstein;
Linearization of regular proper groupoids,
\emph{J. Inst. Math. Jussieu.} \textbf{1} (2002), 493-511.

\bibitem{zung}
N.T. Zung;
Proper groupoids and momentum maps: linearization, affinity and convexity,
\emph{Ann. Scient. \'Ec. Norm. Sup.} $4^e$ s\'erie, \textbf{39} (2006), 841-869.

\end{thebibliography}
\end{document}